\documentclass[a4paper, 10pt]{amsart}
\usepackage[top=80pt,bottom=65pt,left=70pt,right=70pt]{geometry}
\usepackage{graphicx, eepic}
\usepackage{times}
\normalfont
\usepackage{bbm}
\usepackage{xcolor}

\usepackage{amsthm,amsmath,amsfonts,amssymb}
\usepackage{hyperref}
\hypersetup{colorlinks}
\hypersetup{
    bookmarks=true,         
    unicode=false,          
    pdftoolbar=true,        
    pdfmenubar=true,        
    pdffitwindow=false,     
    pdfstartview={FitH},    
    pdftitle={My title},    
    pdfauthor={Author},     
    pdfsubject={Subject},   
    pdfcreator={Creator},   
    pdfproducer={Producer}, 
    pdfkeywords={keyword1} {key2} {key3}, 
    pdfnewwindow=true,      
    colorlinks=true,       
    linkcolor=blue,          
    citecolor=red,        
    filecolor=violet,      
    urlcolor=violet       
}
\usepackage{multirow, url}
\usepackage{color}
\usepackage[all]{xy}
\newtheorem{thm}{Theorem}[section]

\newtheorem{lem}[thm]{Lemma}
\newtheorem{prop}[thm]{Proposition}

\theoremstyle{definition}
\newtheorem{defn}[thm]{Definition}
\theoremstyle{remark}
\newtheorem{rem}[thm]{Remark}

\numberwithin{equation}{section} \numberwithin{table}{section}

\newcommand{\GQ}{{\mathrm{Gal}(\overline{\mathbb{Q}}/{\mathbb{Q}})}}
\newcommand{\GQp}{{\mathrm{Gal}(\overline{\mathbb{Q}}_p/{\mathbb{Q}_p})}}

\newcommand{\inj}{\hookrightarrow}
\newcommand{\surj}{\twoheadrightarrow}
\newcommand{\arsub}{\ar@{}[r]|-*[@]{\subset}}
\newcommand{\arsup}{\ar@{}[r]|-*[@]{\supset}}
\newcommand{\arcap}{\ar@{}[d]|-*[@]{\subset}}
\newcommand{\arcup}{\ar@{}[u]|-*[@]{\subset}}

\renewcommand{\pmod}[1]{{~(\mathrm{mod}~{#1})}}
\renewcommand{\mod}[1]{{~\mathrm{mod}~{#1}}}

\newcommand{\F}{{\mathbb{F}}}

\newcommand{\Q}{{\mathbb{Q}}}
\newcommand{\Z}{{\mathbb{Z}}}

\newcommand{\T}{{\mathbb{T}}}

\newcommand{\bP}{{\mathbb{P}}}

\newcommand{\m}{{\mathfrak{m}}}
\newcommand{\fn}{{\mathfrak{n}}}

\newcommand{\cC}{{\mathcal{C}}}

\newcommand{\cE}{{\mathcal{E}}}

\newcommand{\cI}{{\mathcal{I}}}

\newcommand{\cS}{{\mathcal{S}}}
\newcommand{\cT}{{\mathcal{T}}}

\newcommand{\Res}{{\mathrm{Res}}}

\newcommand{\Gal}{{\mathrm{Gal}}}

\newcommand{\Pic}{{\mathrm{Pic}}}
\newcommand{\End}{{\mathrm{End}}}

\newcommand{\GL}{{\mathrm{GL}}}
\newcommand{\SL}{{\mathrm{SL}}}

\newcommand{\Frob}{{\mathrm{Frob}}}

\newcommand{\old}{{\mathrm{old}}}
\newcommand{\new}{{\mathrm{new}}}

\newcommand{\num}{{\mathrm{num}}}

\newcommand{\modl}{{\pmod {\ell}}} 
\mathchardef\hyp="2D

\newcommand{\br}[1]{\langle #1 \rangle}

\newcommand{\ms}{\hspace*{1cm}}

\newcommand{\zmod}[1]{{\Z/{#1}\Z}}

\newcommand{\exclude}[1]{}

\begin{document}                                                                          

\title{On Eisenstein ideals and the cuspidal group of $J_0(N)$}
\author{Hwajong Yoo}
\address{Center for Geometry and Physics, Institute for Basic Science (IBS), Pohang, Republic of Korea 37673}
\email{hwajong@gmail.com}

\subjclass[2010]{11G18, 14G35}
\keywords{Eisenstein ideals, Cuspidal groups}

\begin{abstract}
Let $\cC_N$ be the cuspidal subgroup of the Jacobian $J_0(N)$ for a square-free integer $N>6$. 
For any Eisenstein maximal ideal $\m$ of the Hecke ring of level $N$, we show that $\cC_N[\m]\neq 0$. 
To prove this, we calculate the index of an Eisenstein ideal $\cI$ contained in $\m$ by computing the order of the cuspidal divisor annihilated by $\cI$. 
\end{abstract}
\maketitle
\setcounter{tocdepth}{1}
\tableofcontents

\section{Introduction} 
Let $N$ be a square-free integer greater than $6$ and let $X_0(N)$ denote the modular curve over $\Q$ associated to $\Gamma_0(N)$, the congruence subgroup of $\SL_2(\Z)$ consisting of upper triangular matrices modulo $N$. There is the Hecke ring $\T:=\T(N)$ of level $N$, which is the subring of the endomorphism ring of the Jacobian variety $J_0(N):=\Pic^0(X_0(N))$ of $X_0(N)$ generated by the Hecke operators $T_n$ for all $n\geq 1$. A maximal ideal $\m$ of $\T$ is called \textit{Eisenstein} if 
the two dimensional semisimple representation $\rho_{\m}$ of $\GQ$ over $\T/\m$ attached to $\m$ is reducible, or equivalently 
$\m$ contains the ideal 
$$
\cI_0(N):=(T_r-r-1~:~ \text{for primes } r \nmid N).
$$ 
Let $\cC_N$ be \textit{the cuspidal group} of $J_0(N)$ generated by degree $0$ cuspidal divisors, which is finite by Manin and Drinfeld \cite{Ma72, Dr73}. 

Ribet conjectured that all Eisenstein maximal ideals are ``cuspidal". In other words, $\cC_N[\m] \neq 0$ for any Eisenstein maximal ideal $\m$. There were many evidences of this conjecture. In particular, special cases were already known (cf. \cite[\textsection 3]{Yoo14}). In this paper, we prove his conjecture.

\begin{thm}[Main theorem]\label{thm:main}
Let $\m$ be an Eisenstein maximal ideal of $\T$. Then $\cC_N[\m] \neq 0$.
\end{thm}

To prove this theorem, we classify all possible Eisenstein maximal ideals in \textsection \ref{sec:Eisenstein}. From now on, we denote by $U_p$ the $p^{\text{th}}$ Hecke operator $T_p \in \T$ when $p \mid N$.
\begin{prop}
Let $\m$ be an Eisenstein maximal ideal of $\T$. Then, it contains
$$
I_{M, N} := ( U_p-1,~U_q-q,~\cI_0(N)~:~\text{for primes } p \mid M~\text{ and}~ q \mid N/M)
$$
for some divisor $M$ of $N$ such that $M\neq 1$.
\end{prop}

In \textsection \ref{sec:cuspidalgroup}, we study basic properties of the cuspidal group $\cC_N$ of $J_0(N)$. In particular, we explicitly compute the order of the cuspidal divisor $C_{M, N}$, which is the equivalence class of $\sum_{\,d \mid M} (-1)^{\omega(d)} P_d$, where $\omega(d)$ is the number of distinct prime divisors of $d$ and $P_d$ is the cusp of $X_0(N)$ corresponding to $1/d \in \bP^1(\Q)$. 

\begin{thm}
The order of $C_{M, N}$ is equal to the numerator of 
$\frac{\varphi(N)\psi(N/M)}{24} \times h$,
where $h$ is either $1$ or $2$. Moreover, $h=2$ if and only if one of the following holds:
\begin{enumerate}
\item $N=M$ and $M$ is a prime such that $M\equiv 1 \pmod {8}$;
\item $N=2M$ and $M$ is a prime such that $M\equiv 1 \pmod {8}$.
\end{enumerate}
\end{thm}
(See Notation \ref{sec:notation} for the definition of $\varphi(N)$ and $\psi(N)$.)
This theorem generalizes the works by Ogg \cite{Og73, Og74} and Chua-Ling \cite{CL97} to the case where $\omega(N)\geq 3$. In \textsection \ref{sec:series}, we introduce Eisenstein series and compute their residues at various cusps. With these computations, we can prove the following theorem in \textsection \ref{sec:index}.

\begin{thm}
If $M\neq N$ and $N/M$ is odd, then the index of ${I_{M, N}}$ is equal to the order of $C_{M, N}$. 
Moreover, if $M=N$ or $N/M$ is even, then the index of ${I_{M, N}}$ is equal to the order of $C_{M, N}$ up to powers of $2$. 
\end{thm}

Finally, combining all the results above, we prove our main theorem in \textsection \ref{sec:proof}.

\subsection*{Acknowledgements}
We are grateful to Ken Ribet for suggesting the problem and his advice during the preparation of this work. We thank the anonymous referee for careful reading and a number of suggestions and corrections to improve the paper.

\subsection{Notation}\label{sec:notation}


For a square-free integer $N=\prod_{i=1}^n p_i $, we define the following quantities:
$$
\omega(N) := n = \text{ the number of distinct prime divisors of } N;
$$
$$
\varphi(N) :=\prod_{i=1}^n (p_i-1) \quad \text{and} \quad \psi(N) :=\prod_{i=1}^n (p_i+1).
$$

For any rational number $x=a/b$, we denote by $\num(x)$ the numerator of $x$, i.e.,
$$
\num(x) := \frac{a}{(a, b)}.
$$

For a prime divisor $p$ of $N$, there is the degeneracy map $\gamma_p: J_0(N/p) \times J_0(N/p) \rightarrow J_0(N)$ (cf. \cite[\textsection 3]{R84}). The image of $\gamma_p$ is called the \textit{$p$-old subvariety} of $J_0(N)$ and is denoted by $J_0(N)_{p\hyp\old}$. The quotient of $J_0(N)$ by $J_0(N)_{p\hyp\old}$ is called the \textit{$p$-new quotient} and is denoted by $J_0(N)^{p\hyp\new}$.
Note that $J_0(N)_{p\hyp\old}$ is stable under the action of Hecke operators and $\gamma_p$ is Hecke-equivariant. Accordingly, the image of $\T(N)$ in $\End(J_0(N)_{p\hyp\old})$ 
(resp. $\End(J_0(N)^{p\hyp\new})$) is called the \textit{$p$-old} (resp. \textit{$p$-new}) \textit{quotient} of $\T(N)$
and is denoted by $\T(N)^{p\hyp\old}$ (resp. $\T(N)^{p\hyp\new}$).  
A maximal ideal $\m$ of $\T(N)$ is called \textit{$p$-old} (resp. \textit{$p$-new}) if its image in  $\T(N)^{p\hyp\old}$ (resp. $\T(N)^{p\hyp\new}$) is still maximal.
Note that if a maximal ideal $\m$ of $\T(N)$ is $p$-old, then there is a maximal ideal $\fn$ of $\T(N/p)$ corresponding to $\m$ (cf. \cite[\textsection 7]{R90}). 

For a prime divisor $p$ of $N$, we denote by $w_p$ the Atkin-Lehner operator (with respect to $p$) acting on $J_0(N)$ (and the space of modular forms of level $N$). (For more detail, see \cite[\textsection 1]{Oh14}.) 

For a prime $p$, we denote by $\Frob_p$ an arithmetic Frobenius element for $p$ in $\GQ$.

\section{Eisenstein ideals}\label{sec:Eisenstein}
From now on, we denote by $N$ a square-free integer greater than $6$ and let $\T:=\T(N)$ be the Hecke ring of level $N$.
A maximal ideal $\m$ of $\T$ is called \textit{Eisenstein} if 
the two dimensional semisimple representation $\rho_{\m}$ of $\GQ$ over $\T/\m$ attached to $\m$ is reducible, or equivalently 
$\m$ contains the ideal $\cI_0(N):=(T_r-r-1~:~ \text{for primes } r \nmid N)$. 
(For the existence of $\rho_{\m}$, see \cite[Proposition 5.1]{R90}.)

Let us remark briefly why these two definitions are equivalent. Let $\m$ be a maximal ideal of $\T$ containing $\ell$. If $\rho_{\m}$ is reducible, then $\rho_{\m} \simeq \mathbbm{1} \oplus \chi_{\ell}$, where $\mathbbm{1}$ is the trivial character and $\chi_{\ell}$ is the mod $\ell$ cyclotomic character, by Ribet \cite[Proposition 2.1]{Yoo14a}. Therefore for a prime $r$ not dividing $\ell N$, we have
$$
T_r \pmod {\m} = \text{trace}(\rho_{\m}(\Frob_r)) = 1 + r 
$$
and hence $T_r-r-1 \in \m$. For $r=\ell$, we get $T_{\ell} \equiv 1+\ell \equiv 1 \pmod {\m}$ by Ribet \cite[Lemma 1.1]{R08}. (This lemma basically follows from the result by Deligne \cite[Theorem 2.5]{Ed92} and this is also true even when $\ell$ divides $N$.) Conversely, if $\m$ contains $\cI_0(N)$, then $\rho_{\m} \simeq \mathbbm{1} \oplus \chi_{\ell}$ by the Chebotarev and the Brauer-Nesbitt theorems.

To classify all Eisenstein maximal ideals, we need to understand the image of $U_p$ in the residue fields for any prime divisor $p$ of $N$.

\begin{lem}\label{lem:eigenvalue}
Let $\m$ be an Eisenstein maximal ideal of $\T$. Let $p$ be a prime divisor of $N$ and 
$U_p-\epsilon(p) \in \m$. Then, $\epsilon(p)$ is either $1$ or $p$ modulo $\m$.
\end{lem}

\begin{proof}
Assume that $\m$ is $p$-old. Then $\m$ can be regarded as a maximal ideal of $\T^{p\hyp\old}$.  
Let $R$ be the common subring of the Hecke ring $\T(N/p)$ of level $N/p$ and $\T^{p\hyp\old}$, which is generated by all $T_n$ with $p \nmid n$. 
Let $\fn$ be the corresponding maximal ideal of $\T(N/p)$ to $\m$ and $T_p$ be the $p^{\text{th}}$ Hecke operator in $\T(N/p)$. 
Then, we get
$$
\T(N/p) = R [T_p] \quad\text{and}\quad \T(N)^{p\hyp\old} = R[U_p]
$$
\cite[\textsection 7]{R90} and $\T/\m\simeq \T(N/p)/\fn$.
Two operators $T_p$ and $U_p$ are connected by the quadratic equation
$U_p^2-T_p U_p+p=0$ (\textit{loc. cit.}). Note that $T_p-p-1 \in \fn$ because $\fn$ is Eisenstein as well.
Therefore over the ring $\T/\m \simeq \T(N/p)/\fn$, we get
$U_p^2-(p+1)U_p+p = (U_p-1)(U_p-p) = 0$ and hence either $\epsilon(p) \equiv 1$ or $p \pmod {\m}$.

Assume that $\m$ is $p$-new. Then $\epsilon(p) = \pm 1$. Therefore it suffices to show that $\epsilon(p) \equiv 1$ or $p \pmod {\m}$ when $\epsilon(p)=-1$. 
Let $\ell$ be the residue characteristic of $\m$. 
If $\ell=2$, then there is nothing to prove because $1 \equiv -1 \pmod {\m}$. 
If $\ell=p$, then $U_p \equiv 1 \pmod {\m}$ by Ribet \cite[Lemma 1.1]{R08}. Therefore we assume that $\ell\geq 3$ and $\ell \neq p$.
On the one hand, we have $\rho_{\m} \simeq \mathbbm{1} \oplus \chi_{\ell}$. On the other hand, the semisimplification of the restriction of $\rho_{\m}$ to $\GQp$ is isomorphic to $\epsilon \oplus \epsilon\chi_{\ell}$, where $\epsilon$ is the unramified quadratic character with $\epsilon(\Frob_p)=\epsilon(p)$ because $\m$ is $p$-new (cf. \cite[Theorem 3.1.(e)]{DDT}). Since $\epsilon(p)=-1$, we get $p \equiv -1 \modl$ and hence $\epsilon(p) \equiv p \pmod {\m}$.
\end{proof}

Let $\m$ be an Eisenstein maximal ideal of $\T$ containing $\ell$. Then, it contains 
$$
I_{M, N} := (U_p-1,~U_q-q,~\cI_0(N)~:~\text{for primes } p \mid M ~~\text{ and } ~~q \mid N/M) \subseteq \T
$$
for some divisor $M$ of $N$ by the previous lemma. If $q \equiv 1 \modl$ for a prime divisor $q$ of $N/M$, then $\m=(\ell, ~I_{M, N})=(\ell,~I_{M\times q, N})$. Therefore when we denote by $\m:=(\ell, ~I_{M, N})$ for some divisor $M$ of $N$, we always assume that $q \not\equiv 1 \modl$ for all prime divisors $q$ of $N/M$. Hence if $\ell=2$, then either $\m:=(\ell,~I_{N, N})$ or $\m:=(\ell,~I_{N/2, N})$. If $\ell\geq 3$, $\m:=(\ell,~I_{1, N})$ cannot be maximal by Proposition \ref{prop:nonmaximal}. Therefore from now on, we always assume that $M\neq 1$.

\section{The cuspidal group} \label{sec:cuspidalgroup}
As before, let $N$ denote a square-free integer and let $M\neq 1$ denote a divisor of $N$.
For a divisor $d$ of $N$, we denote by $P_d$ the cusp corresponding to $1/d$ in $\bP^1(\Q)$. (Thus, the cusp $\infty$ is denoted by $P_N$.) 
We denote by $C_{M, N}$ the equivalence class of a cuspidal divisor $\sum_{\,d\mid M} (-1)^{\omega(d)} P_d$. Note that $I_{M, N}$ annihilates $C_{M, N}$ \cite[Proposition 2.13]{Yoo14}. To compute the order of $C_{M, N}$, we use the method of Ling \cite[\textsection 2]{Lg97}.  

\begin{thm}\label{thm:ordercuspidal}
The order of $C_{M, N}$ is equal to
$$
 \num\left(\frac{\varphi(N)\psi(N/M)}{24} \right)\times h,
$$
where $h$ is either $1$ or $2$. Moreover, $h=2$ if and only if one of the following holds:
\begin{enumerate}
\item $N=M$ and $M$ is a prime such that $M\equiv 1 \pmod {8}$;
\item $N=2M$ and $M$ is a prime such that $M\equiv 1 \pmod {8}$.
\end{enumerate}
\end{thm}

\begin{rem}
The size of the set $\cC_N$ is computed by Takagi \cite{Ta97}. Recently, Harder discussed the more general question of giving denominators of Eisenstein cohomology classes. The order of a cuspidal divisor is a special case of such a denominator and some cases were computed by a slightly different method from the one used here \cite[\textsection 2]{Ha15}.
\end{rem}

Before starting to prove this theorem, we define some notations and provide lemmas.

Let $N=\prod_{i=1}^n p_i$. We denote by $\cS$ the set of divisors of $N$. Let $s:=2^n=\# \cS$. 
\begin{enumerate}
\item For $a\in \cS$, we denote by 
$$
a=(a_1, \,a_2, \cdots, \,a_n),
$$
where $a_i=0$ if $(p_i, \,a)=1$; and $a_i=1$ otherwise. For instance, $1=(0, \,0, \cdots, \,0)$ and
$N=(1, \,1, \cdots, \,1)$. 

\item We define the total ordering on $\cS$ as follows.

Let $a, b \in \cS$ and $a\neq b$. 
\begin{itemize}
\item If $\omega(a)<\omega(b)$, then $a<b$. In particular, $1<a<N$ for $a \in \cS\setminus \{1, N\}$.
\item If $\omega(a)=\omega(b)$, then we use the anti-lexicographic order. In other words, $a<b$ if $a_i=b_i$ for all $i<t$ and $a_t > b_t$.
\end{itemize}

\item We define the box addition $\boxplus$ on $\cS$ as follows.
$$
a \boxplus b:=(c_1, \,c_2, \cdots, \,c_n),
$$
where $c_i \equiv a_i + b_i +1 \pmod {2}$ and $c_i \in \{0, 1\}$. For instance, $p_1\boxplus p_1=N$ and $1 \boxplus a=N/a$.

\item Finally, we define the sign on $\cS$ as follows. 
$$
\mathrm{sgn}(a):=(-1)^{s(a)},
$$
where $s(a)=\omega(N)-\omega(a)$.
For example, $\mathrm{sgn}(N)=1$ and $\mathrm{sgn}(1)=(-1)^n$.

\end{enumerate}

We denote by $\cS = \{d_1, ~d_2,~\dots,~d_s \}$, where $d_i < d_j$ if $i <j$. For instance, $d_1=1$, $d_2=p_1$ and $d_s=N$. Note that $d_i \times d_{s+1-i}=N$ for any $i$. 

For ease of notation, we denote by $d_{ij}$ the box sum $d_i \boxplus d_j$. 
\begin{lem}\label{lem:boxproperties}
We have the following properties of  $~\boxplus$.
\begin{enumerate}
\item $d_{ij} = d_{ji} = d_{s+1-i} \boxplus d_{s+1-j}$.
\item $d_{i1}=N/{d_i}=d_{s+1-i}$.
\item\label{boxproperty3} $\cS = \{d \boxplus d_1, ~d \boxplus d_2, ~\dots,~d \boxplus d_s \}$ for any $d=d_i$.
\item $\mathrm{sgn}(d_{ij})=\mathrm{sgn}(d_i) \times \mathrm{sgn}(d_j)$.
\item\label{boxproperty5} Assume that $i \neq j$ and $d_{ij}$ is not divisible by $p_n$. Then, for any $d_k$ such that $d_{kj}$ is not divisible by $p_n$, we get
$$
d_{i k} \times d_{k j} = d_{i r(k)} \times d_{r(k) j},
$$
where $r(k)$ is the unique integer between $1$ and $s$ such that $d_{r(k) j}=p_n \cdot d_{kj}$. 
\end{enumerate}
\end{lem}
\begin{proof}
The first, second, third and fourth assertions easily follow from the definition. 

Assume that $i \neq j$. Then $d_{ij} \neq N$ and there is a prime divisor of $N/{d_{ij}}$. 
Assume that $d_{ij}$ is not divisible by $p_n$.
Let $k$ be an integer such that $d_{kj}$ is not divisible by $p_n$. Then, we denote by
$$
d_i = (a_1, \cdots, \,a_n)  \quad \text{and}\quad d_j=(b_1, \cdots, \,b_n);
$$
$$
d_k = (c_1, \cdots, \,c_n) \quad \text{and}\quad d_{r(k)}=(e_1, \cdots, \,e_n).
$$

By abuse of notation, we denote by $d_{ik} \times d_{kj} = (x_1, \,x_2, \cdots, \,x_n)$ and $d_{i r(k)} \times d_{r(k) j}=(y_1, \,y_2, \cdots, \,y_n)$, where $0 \leq x_t,~y_t \leq 2$. Thus, $d_{ik} \times d_{kj}=\prod_{t=1}^s p_t^{x_t}$. It suffices to show that $x_t = y_t$ for all $t$.
\begin{itemize}
\item Assume that $t\neq n$. From the definition of $d_{r(k)}$, we get $c_t = e_t$. Therefore $x_t = y_t$.
\item  Since $d_{ij}$ and $d_{kj}$ is not divisible by $p_n$, we get $a_n+b_n=1=c_n+b_n$. Therefore $a_n=c_n$. Since $d_{r(k)j}$ is divisible by $p_n$, we get $e_n+b_n+1\equiv 1 \pmod {2}$. Therefore $x_n=y_n=1$. 
\end{itemize}
\end{proof}

From now on, we follow the notations in \cite[\textsection 2]{Lg97}.
In our case, the $s\times s$ matrix $\Lambda$ on page 35 of \textit{op. cit.} is of the form 
$$
\Lambda_{ij}=\frac{1}{24} a_N(d_i, \,d_j),
$$
where
$$
a_N(a, \,b):=\frac{N}{(a, N/a)}\frac{(a, \,b)^2}{ab}.
$$
For examples, $a_N(1, \,p)=N/p$ and $a_N(N, \,p)=p$.
\begin{lem}
We get
$$
24\times \Lambda_{ij}= d_i \boxplus d_j =d_{ij} \in \cS.
$$
\end{lem}
\begin{proof}
This is clear from the definition.
\end{proof}

\begin{lem}
Let $A:=(\mathrm{sgn}(d_{ij})\times (d_{ij}))_{1\leq i,\, j \leq s}$ be a $s \times s$ matrix. Then, $A=\frac{\varphi(N)\psi(N)}{24} \times \Lambda^{-1}$.
\end{lem}
\begin{proof}
We compute $B:=24 \times \Lambda \times A$. 
\begin{itemize}
\item Assume that $i=j$. Then, we have
$$
B_{ii}=\sum_{j=1}^s \mathrm{sgn}(d_{ij}) \times (d_{ij})^2=\sum_{k=1}^s \mathrm{sgn}(d_k)\times d_k^2 = \prod_{k=1}^n (p_k^2-1)=\varphi(N)\psi(N)
$$
because $\{ d_{ij} ~:~1\leq j \leq s \} = \cS$ by Lemma \ref{lem:boxproperties} (\ref{boxproperty3}).

\item Assume that $i \neq j$. Then, $d_{ij} \neq N$. Let $q$ be a prime divisor of $N/{d_{ij}}$. We denote by $\cT_j$ the subset of $\cS$ such that 
$$
\cT_j := \{ d_k \in \cS~:~(q, \,d_{kj})=1 \}.
$$
Then the size of $\cT_j$ is $s/2$. For each element $d_k \in \cT_j$, we can find $d_{r(k)} \in \cS$ such that $d_{r(k)j}=q \cdot d_{kj}$ by Lemma \ref{lem:boxproperties} (\ref{boxproperty3}). Moreover ${\cT_j}^c=\{ d_{r(k)} ~:~ d_k \in \cT_j  \}$ and we get $\mathrm{sgn}(d_{r(k)j})=-\mathrm{sgn}(d_{kj})$. For each $d_k \in \cT_j$, we get $d_{i k} \times d_{k j} = d_{i r(k)} \times d_{r(k) j}$ by Lemma \ref{lem:boxproperties} (\ref{boxproperty5}).
Therefore, we have
$$
B_{ij}=\sum_{k=1}^s \mathrm{sgn}(d_{kj}) (d_{ik} \times d_{kj})
=\sum_{d_k \in \cT_j } \mathrm{sgn}(d_{kj})  \left[ (d_{ik} \times d_{kj})-(d_{ir(k)} \times d_{r(k)j}) \right]=0.
$$
\end{itemize}
\end{proof}

The matrix form of $C_{M, N}$ in the set $S_2$ on \cite[P. 34]{Lg97} is then
$$
\text{ for } 1 \leq a \leq s,\quad (C_{M, N})_{a1}= \begin{cases}
(-1)^{\omega(d_a)}=(-1)^n \times {\mathrm{sgn}(d_a)} \!\!\quad\text{ if }~~~d_a \mid M, \\ 
~~\quad\ms 0 \quad\quad\ms\ms\quad\text{ otherwise}. \\
\end{cases}
$$

Finally, we prove the following lemma.
\begin{lem}
Let $E:=\Lambda^{-1}C_{M, N}$. Then for $1\leq a \leq s$ we have
$$
E_{a1}
=\mathrm{sgn}(d_{s+1-a})\times \frac{24}{\varphi(N)\psi(N/M)} \times \frac{d_{s+1-a}}{(d_{s+1-a},~M)}.
$$
In particular, $E_{s1}=(-1)^{\omega(N)}\frac{24}{\varphi(N)\psi(N/M)}$. Moreover if $M=N$, then we get
$$
E_{a1}=\mathrm{sgn}(d_{s+1-a}) \times \frac{24}{\varphi(N)}.
$$
\end{lem}
\begin{proof}
Let $D:=d_{s+1-a}=N/d_a$ and $E:=(D, \,M)$. Then, by direct calculation we have
$$
d_{ar}=d_a \boxplus d_r = \frac{D \times d_r}{(D, \,d_r)^2}
$$
and the sign of $(\Lambda^{-1})_{ak} \times (C_{M, N})_{k1}$ is $\mathrm{sgn}(d_a)\times \mathrm{sgn}(d_k)\times (-1)^n \times \mathrm{sgn}(d_k) = \mathrm{sgn}(D)$ for any divisor $d_k$ of $M$.
Therefore we have
\begin{equation*}
\begin{split}
\sum_{k=1}^s \mathrm{sgn}(d_{ak}) \times d_{ak} \times  (C_{M, \, N})_{k1}
&=\mathrm{sgn}(D) \times \sum_{d_r \mid M} \frac{D \times d_r}{(D, \,d_r)^2} \\
&= \mathrm{sgn}(D) \times \frac{D}{E} \times \sum_{d_r \mid M} \frac{E\times d_r}{(E,\,d_r)^2}.
\end{split}
\end{equation*}
We denote by  
$$
D_r:=\frac{E\times d_r}{(E,\,d_r)^2}=\frac{(D, \,M)\times d_r}{((D, \,M),~d_r)^2}.
$$ 
Then, $D_r$ is a divisor of $M$ and for two distinct divisors $d_{r_1},~d_{r_2}$ of $M$, we get $D_{r_1} \neq D_{r_2}$. Therefore, we have
$$
\sum_{d_r \mid M} \frac{E\times d_r}{(E,~d_r)^2}=\sum_{d_r \mid M} D_r=\sum_{d \mid M} d = \psi(M),
$$
which implies the result.
\end{proof}

Now we give a proof of the theorem above.

\begin{proof}[Proof of Theorem \ref{thm:ordercuspidal}]
We check the conditions in Proposition 1 in \textit{op. cit.} (We use the same notations.)
\begin{itemize}
\item The condition (0) implies that the order of $C_{M, N}$ is of the form $\frac{\varphi(N)\psi(N/M)}{24}\times g$ for some integer $g\geq 1$.
\item The condition (1) always holds unless $M=N$ because $\sum_{\delta \mid N} r_{\delta} \cdot \delta=0$. If $M=N$, then $\sum_{\delta \mid N} r_{\delta} \cdot \delta=(-1)^{n}g\varphi(N)\equiv 0 \pmod {24}$.
\item The condition (2) implies that $g=\num(\frac{24}{\varphi(N)\psi(N/M)}) \times h$ for some integer $h\geq 1$ because $\sum_{\delta \mid N} r_{\delta} \cdot N/\delta = g\varphi(N)\psi(N/M) \equiv 0 \pmod {24}$. 
\item The condition (3) always holds.
\item The condition (4) always holds unless $M$ is a prime because 
$\prod_{\delta \mid N} \delta^{r_{\delta}}=1$. If $M$ is a prime, then it implies that $g\varphi(N/M)$ is even because $\prod_{\delta \mid N} \delta^{r_{\delta}} = M^{-g\varphi(N/M)}$.
\end{itemize}

In conclusion, the order of $C_{M, N}$ is equal to $\num(\frac{\varphi(N)\psi(N/M)}{24})\times h$ for the smallest positive integer $h$ satisfying all the conditions above. Therefore we get $h=1$ unless all the following conditions hold:
\begin{enumerate}
\item $M$ is a prime; 
\item $\varphi(N/M)=1$;
\item $\num(\frac{24}{\varphi(N)\psi(N/M)})$ is odd.
\end{enumerate}
Moreover if all the conditions above hold, then $h=2$. By the first condition, $M$ is a prime.
By the second condition, either $N=M$ or $N=2M$. 
\begin{itemize}
\item Assume that $N=M$ is a prime greater than $3$. Then, $h=2$ if and only if $M\equiv 1\pmod {8}$. This is proved by Ogg \cite{Og73}.
\item Assume that $N=2M$. Then, $h=2$ if and only if $M \equiv 1 \pmod {8}$. This is proved by Chua and Ling \cite{CL97}.
\end{itemize} 
\end{proof}

\section{Eisenstein series}\label{sec:series}
As before, let $N=\prod_{i=1}^n p_i$ and $M=\prod_{i=1}^m p_i$ for $1\leq m\leq n$.
Let 
$$
e(z):=1-24\sum_{n \geq 1} \sigma(n) \times q^n
$$ 
be the $q$-expansion of \textit{Eisenstein series of weight $2$} of level $1$ as on \cite[p.~78]{M77}, where $\sigma(n)=\sum_{d\mid n} d$ and $q = e^{2 \pi i z}$.  
\begin{defn}
For any modular form $g$ of weight $k$ and level $A$; and a prime $p$ not dividing $A$, we define 
modular forms $[p]_k^+(g)$ and $[p]_k^-(g)$ of weight $k$ and level $pA$ by
$$
[p]_k^+(g)(z):= g(z) - p^{k-1} g(pz) \quad\text{and} \quad [p]_k^-(g)(z):= g(z) - g(pz).
$$ 
Using these operators, we define Eisenstein series of weight 2 and level $N$ by
$$
\cE_{M, N} (z):=[p_n]_2^- \circ \cdots \circ [p_{m+1}]_2^- \circ [p_m]_2^+ \circ \cdots \circ [p_1]_2^+(e)(z).
$$
(Note that $\cE_{M, N}=-24 E_{M, N}$, where $E_{M, N}$ is a normalized Eisenstein series in \cite[\textsection 2.2]{Yoo14}.)
\end{defn}

By Proposition 2.6 of \textit{op. cit.}, we know that $\cE_{M, N}$ is an eigenform for all Hecke operators and $I_{M, N}$ annihilates $\cE_{M, N}$. By Proposition 2.10 of \textit{op. cit.}, we can compute the residues of $\cE_{M, N}$ at various cusps. 

\begin{prop}\label{prop:residue}
We have
$$
\Res_{P_N} (\cE_{M, N})= \begin{cases}
(-1)^{n}\varphi(N) \!\!\!\quad\text{ if }~~M=N, \\ 
~~\quad\quad 0 ~\ms\quad~\text{otherwise}.
\end{cases} 
$$
Moreover, for a prime divisor $p$ of $N$ we have
$$
\Res_{P_{N/p}}(\cE_{N, N})=(-1)^{n-1}\varphi(N) \quad\text{and}\quad\Res_{P_M} (\cE_{M, N})=(-1)^{\omega(M)}\varphi(N)\psi(N/M)(M/N).
$$
\end{prop}
\begin{proof}
The first statement follows from the definition (cf. \cite[\textsection II.5]{M77}). For the second statement,
we use the method of Deligne-Rapoport \cite{DR73} (cf. 3.17 and 3.18 in \textsection VII.3) or of Faltings-Jordan \cite{FJ95} (cf. Proposition 3.34). Therefore the residue of $\cE_{M, N}$ at $P_1$ is $\varphi(N)\psi(N/M)(M/N)$ (cf. \cite[Proposition 2.11]{Yoo14}). 
Since the Atkin-Lehner operator $w_p$ acts by $-1$ on $\cE_{M, N}$ for a prime divisor $p$ of $M$, $w_M$ acts by $(-1)^{\omega(M)}$ and hence the result follows.
\end{proof}

\section{The index of an Eisenstein ideal}\label{sec:index}
As before, let $N=\prod_{i=1}^n p_i$ and $M=\prod_{i=1}^m p_i$ for some $1\leq m\leq n$. Let $\T:=\T(N)$.

Note that $\T/{I_{M, N}} \simeq \zmod {t}$ for some integer $t\geq 1$ \cite[Lemma 3.1]{Yoo14}. We compute the number $t$ as precise as possible.  

\begin{thm}\label{thm:index1}
The index of ${I_{N, N}}$ is equal to the order of $C_{N, N}$ up to powers of $2$. 
\end{thm}

\begin{thm}\label{thm:index2}
If $M\neq N$ and $N/M$ is odd (resp. even), then the index of ${I_{M, N}}$ and the order of $C_{M, N}$ coincide (resp. coincide up to powers of $2$). 
\end{thm}

Before starting to prove the theorems, we introduce some notations. 
\begin{defn}
For a prime $\ell$, we define $\alpha(\ell)$ and $\beta(\ell)$ as follows:
$$
\left( \T/{I_{M, N}} \right)\otimes_{\Z} \Z_{\ell} \simeq \zmod {\ell^{\alpha(\ell)}} \quad \text{and}
$$
$$
\ell^{\beta(\ell)} \text{ is the exact power of } \ell \text{ dividing } \num\left( \frac{\varphi(N)\psi(N/M)}{24} \times h\right),
$$
where $h$ is the number in Theorem \ref{thm:ordercuspidal}.
\end{defn}

Since $I_{M, N}$ annihilates $C_{M, N}$, we get $\alpha(\ell) \geq \beta(\ell)$ (cf. \cite[proof of Theorem 3.2]{Yoo14}). Therefore to prove Theorems \ref{thm:index1} and \ref{thm:index2},  
it suffices to show that $\alpha(\ell) \leq \beta(\ell)$ for all (or odd) primes $\ell$.  If $\alpha(\ell)=0$, then there is nothing to prove. Thus, we now assume that $\alpha(\ell)\geq 1$. Let 
$$
\cI:=(\ell^{\alpha(\ell)},~{I_{M, N}})
$$ 
and let $\delta$ be a cusp form of weight $2$ and level $N$ over the ring $\T/\cI \simeq \zmod {\ell^{\alpha(\ell)}}$ whose $q$-expansion (at $P_N$) is 
$$
\sum\limits_{n\geq 1} (T_n \mod \cI) \times q^n.
$$

Now we prove the theorems above.
\begin{proof}[Proof of Theorem \ref{thm:index1}]
First, let $\ell=3$ and $M=N$. 
Let $E:=\cE_{N, N} \pmod {3^{\alpha(3)+1}}$ and $A=(-1)^{\omega(N)} \varphi(N)$.
Since $24\delta$ is a cusp form of weight $2$ modulo $3^{\alpha(3)+1}$ (cf. \cite[p.~86]{M77}), 
$E+24\delta$ is a modular form of weight $2$ and level $N$ over $\zmod {3^{\alpha(3)+1}}$. 
Let $a=\min\{\alpha(3),~\beta(3)+1\}$. Then, by the $q$-expansion principle \cite[\textsection 1.6]{Ka73} we have
$$
E+24 \delta \equiv Ae \pmod {3^{a+1}}
$$
on the irreducible component $C$ of $X_0(N)_{\F_{\ell}}$ containing $P_N$
because $Ae$ is a modular form of weight $2$ over $\zmod {(12A)}$ and  $(3^{\alpha(3)+1},~12A)=3^{a+1}$. 
By the following lemma, we get $A \equiv 0 \pmod 3$ and hence we can choose a prime divisor $p$ of $N$ congruent to $1$ modulo $3$.
Note that the cusp $P_{N/p}$ belongs to $C$. By Proposition \ref{prop:residue}, $\Res_{P_{N/p}}(E)=-A$ and $\Res_{P_{N/p}} (Ae) \equiv pA \pmod {12A}$ by Sublemma on \cite[p.~86]{M77}. 
Combining all the computations above, we have
$$
 \Res_{P_{N/p}} (8\delta) \equiv \frac{(p+1)A}{3} \pmod {3^a}.
$$
Since $\delta$ is a cusp form modulo ${3^{\alpha(3)}}$, we get $\Res_{P_{N/p}} (8\delta) \equiv 0 \pmod {3^{\alpha(3)}}$ and hence $3^{\beta(3)} \equiv 0 \pmod {3^{\alpha(3)}}$. In other words, 
we get $\alpha(3)\leq \beta(3)$.

Next, let $\ell \geq 5$ and $M=N$. Let $F:=\cE_{N, N} \pmod {\ell^{\alpha(\ell)}}$. Then, $f:=F+24\delta$ is a modular form of weight $2$ and level $N$ over $\zmod {\ell^{\alpha(\ell)}}$ whose $q$-expansion is $A$. Basically the inequality $\alpha(\ell)\leq \beta(\ell)$ follows from the non-existence of a mod $\ell$ modular form of weight $2$ and leven $N$ whose $q$-expansion is a non-zero constant (cf. \cite[chap. II, Proposition 5.6]{M77} and \cite[Proposition (2.2.6)]{Oh14}).

\begin{itemize}
\item
If $\ell \nmid N$, then by Ohta \cite[Lemma (2.1.1)]{Oh14}, we can find a modular form $g$ of weight $2$ and level $1$ such that $f(z)=g(Nz)$. Therefore $A \equiv 0 \pmod {\ell^{\alpha(\ell)}}$ (cf. \cite[chap. II, Proposition 5.6]{M77}) and hence we get $\alpha(\ell)\leq \beta(\ell)$. 

\item
Assume that $\ell\mid N$ and $\m:=(\ell,~\cI)$ is not $\ell$-new. Then, the argument basically follows from the previous case because the exact powers of $\ell$ dividing $A$ and $\varphi(N/{\ell})$ coincide. (For more detailed argument on lowering the level when $\ell \geq 5$, see the proof of Theorem \ref{thm:index2} below.)

\item 
Assume that $\ell \mid N$ and $\m:=(\ell,~\cI)$ is $\ell$-new. 
Then, we can lift $\delta$ to a modular form $\widetilde{\delta}$ of weight $2$ and level $N$ over $\Z_{(\ell)}$ satisfying $w_{\ell}(\widetilde{\delta})=-\widetilde{\delta}$, where 
$\Z_{(\ell)}$ is the localization of $\Z$ at $\ell$. 
Therefore $\widetilde{\delta}$ determines a regular differential on $X_0(N)_{\Z_{(\ell)}}$ over $\Z_{(\ell)}$ (cf. \cite[Proposition (1.4.9)]{Oh14}). 
Similarly, we can lift $F$ to $\cE_{N, N}$ as well and $w_{\ell}(\cE_{N, N})=-
\cE_{N, N}$.
Therefore 
$f=\cE_{N, N} +24\widetilde{\delta} \pmod {\ell^{\alpha(\ell)}}$ can be regarded as a regular differential on $X_0(N)_{\Z_{(\ell)}}$ over $\zmod {\ell^{\alpha(\ell)}}$ whose $q$-expansion is $A$. 
If $\alpha(\ell)\geq \beta(\ell)+1$, then $g=f \pmod {\ell^{\beta(\ell)+1}}$ is a regular differential over $\zmod {\ell^{\beta(\ell)+1}}$. Moreover $\ell^{-\beta(\ell)} \times g$ can be regarded as a regular differential over $\F_{\ell}$ whose $q$-expansion is a non-zero constant (cf. \cite[p. 86]{M77}), which is a contradiction (cf. \cite[Proposition (2.2.6)]{Oh14}). Thus, we get $\alpha(\ell)\leq \beta(\ell)$. 
\end{itemize}
\end{proof}

\begin{lem}\label{lem:maximal 3 A=0}
If $\m:=(3,~I_{N, N})$ is maximal, then $A=(-1)^{\omega(N)} \varphi(N) \equiv 0 \pmod 3$.
\end{lem}
\begin{proof}
As above, let $E:=\cE_{N, N} \pmod {9}$ and $\eta:=\delta \pmod {\m}$. Let $f:=E+24\eta$ be a modular form of weight $2$ and level $N$ over $\zmod 9$ whose $q$-expansion is $A$.

First, assume that $3$ does not divide $N$. Then by Ohta \cite[Lemma (2.1.1)]{Oh14}, we can find a modular form $g$ of weight $2$ and level $1$ over $\zmod 9$ such that $f(z)=g(Nz)$. By Mazur \cite[chap. II, Proposition 5.6]{M77}, we get $A \equiv 0 \pmod 3$. 

Next, assume that $p_1=3$ and $N=3M$. If $\m$ is $3$-old, then the result follows from the previous case. Thus, we further assume that $\m$ is $3$-new. 
Then as above, we can regard $\eta$ as a regular differential on $X_0(N)_{\Z_{(\ell)}}$ over $\F_3$ and hence 
there is a modular form $\zeta$ of weight $3+1$ and level $M$ over $\F_3$ which has the same $q$-expansion as $\eta$ by Ohta \cite[Proposition (2.2.4)]{Oh14}. By the same argument as on \cite[p. 86]{M77}, $240 \zeta$ is a modular form of weight $4$ and level $M$ over $\zmod 9$.
Let $E_4$ be the usual Eisenstein series of weight $4$ and level $1$:
$$
E_4 (z)= 1 + 240 \sum\limits_{n=1}^\infty \sigma_3(n) \times q^n,
$$
where $\sigma_3(n)=\sum_{d \mid n} d^3$ and $q=e^{2\pi i z}$. 
Let $G(z):=[p_n]_4^+ \circ \cdots \circ [p_2]_4^+(E_4)(z)$ be an Eisenstein series of weight $4$ and level $M$ whose constant term is $\prod_{i=2}^n (1-p_i^3)$.
Now we consider the modular form $h:=G \pmod 9 - 240 \zeta$ of weight $4$ and level $M$ over $\zmod 9$. Since the $q$-expansion of 
$h$ is $\prod_{i=2}^n (1-p_i^3)$, there is a modular form $H$ of weight $4$ and level $1$ over $\zmod 9$ such that $h(z)=H(Mz)$ by Ohta \cite[Lemma (2.1.1)]{Oh14}.
However if $A\not\equiv 0 \pmod 3$, then there is no such a modular form over $\zmod 9$ (cf. \cite[p. 308]{Oh14}) because $1-p_i^3 \equiv 1-p_i \pmod 3$. Therefore we get $A \equiv 0 \pmod 3$.
\end{proof}

\begin{proof}[Proof of Theorem \ref{thm:index2}]
Since we assume that $\alpha(\ell)\geq 1$, $\m:=(\ell,~I_{M, N})$ is maximal.

First, assume that $N/M$ is divisible by an odd prime $\ell$. Then $U_{\ell}\equiv \ell \equiv 0 \pmod {\m}$ and hence $\m$ is not $\ell$-new.
Thus, we get $\T(N)/\cI \simeq \T(N)^{\ell\hyp\old}/\cI$. 
Let $R$ be the common subring of $\T(N/{\ell})$ and $\T(N)^{\ell\hyp\old}$, which is generated by all $T_n$ with $\ell \nmid n$.  Then, 
as in the proof of Lemma \ref{lem:eigenvalue}, 
$\T(N/\ell) = R [T_\ell]$ and $\T(N)^{\ell\hyp\old} = R[U_\ell]$.
Note that if $\ell$ is odd then $R=\T(N/\ell)$ by Ribet \cite[p. 491]{Wi95} and $\T(N)^{\ell\hyp\old}
\simeq R[X]/{(X^2-T_{\ell}X+\ell)}$. Let $I$ be the ideal of $R$ generated by all the generators of $\cI$ but $U_{\ell}-\ell$. Then, we show that $T_{\ell}-\ell-1 \in I$ as follows.
Note that the kernel $K$ of the composition of the maps 
$$
R=\T(N/{\ell}) \inj \T(N)^{\ell\hyp\old}=R[U_{\ell}]/(U_{\ell}^2-T_{\ell}U_{\ell}+\ell) \surj \T(N)^{\ell\hyp\old}/{\cI}\simeq \zmod {\ell^{\alpha(\ell)}}
$$ 
(sending $T_n$ to $T_n \pmod {\cI}$) is $(I, ~\ell(T_{\ell}-\ell-1))$ and this composition is clearly surjective. Thus,
we get $R/I \surj R/K$. Since all the generators of $R$ are congruent to integers modulo $I$; and $I$ contains $\ell^{\alpha(\ell)}$, we have $R/I = R/K \simeq \zmod {\ell^{\alpha(\ell)}}$; in particular $\ell(T_{\ell}-\ell-1) \in I$.
Let $f$ be a cusp form over $R/I$ whose $q$-expansion is
$\sum_{n\geq 1} (T_n \mod {I})\times q^n$. 
\begin{itemize}
\item
Suppose that $\ell \geq 5$. Let $E:=\cE_{M,\,{N/{\ell}}} \pmod {\ell^{\alpha(\ell)}}$ and let $g:=24f+E$. Then, $g$ is a modular form over $R/I \simeq \zmod {\ell^{\alpha(\ell)}}$ whose $q$-expansion is of the form $\sum_{k\geq 0} a_k \times q^{\ell k}$. 
By Katz \cite[Corollaries (2) and (3) of the main theorem]{Ka76}, we get $g=0$ and hence
$a_1/{24}=T_{\ell}-\ell-1 \in I$. (Note that the constant term $a_0$ must be $0$ and hence we get $\alpha(\ell)\leq \beta(\ell)$ as well if $M=N/{\ell}$.)

\item
Suppose that $\ell=3$. Let $E:=\cE_{M,\,{N/{\ell}}} \pmod {3^{\alpha(3)+1}}$ and let $g:=24f+E$.
Then, $g$ is a modular form over $\zmod {3^{\alpha(3)+1}}$ whose $q$-expansion is of the form
$\sum_{k\geq 0} a_k  \times q^{\ell k}$. 
If $a_1 =0 \in \zmod {3^{\alpha(3)+1}}$ then $a_1/24 = T_3 - 4 = 0 \in \zmod {3^{\alpha(3)}} \simeq R/I$ and hence $T_3-4 \in I$.
Therefore it suffices to show that $a_1 = 0 \in \zmod {3^{\alpha(3)+1}}$.

If $M \neq N/{\ell}$ then $a_0=0$ and hence $g=0$ by Corollaries (3) and (4) in \textit{loc. cit.} Therefore $a_1=0$.

Suppose that $M=N/{\ell}$. Then, $a_0= (-1)^{\omega(M)} \varphi(M)$. Note that the exact power of $3$ dividing $a_0$ is $3^{\beta(3)+1}$.
Since $3(T_3-4) \in I$, $g \pmod {3^{\alpha(3)}}$ is a modular form over $\zmod {3^{\alpha(3)}}$ whose
$q$-expansion is a constant $a_0$. Since $a_0 \times e$ is a modular form over $\zmod {3^{\beta(3)+2}}$ whose $q$-expansion is $a_0$, by the $q$-expansion principle $g=24f+E \equiv a_0 \times e \pmod {3^{a}}$, where $a=\min \{\alpha(3),\, \beta(3)+2 \}$.
Since $\m$ is $\ell$-old, there is the corresponding maximal ideal $\fn$ of $\T(N/\ell)$ to $\m$. Hence by Lemma \ref{lem:maximal 3 A=0}, $a_0 \equiv 0 \pmod 3$ and we can find a prime divisor $p$ of $N/{\ell}$ such that $p\equiv 1 \pmod 3$.
By comparing the residues of $g$ and $a_0 \times e$ at $P_{N/p}$ as in the proof of Theorem \ref{thm:index1}, we get $(p+1)a_0 \equiv 0 \pmod {3^a}$ and hence $\alpha(3) \leq \beta(3)+1$. Therefore $h:=3^{-\alpha(3)}\times g$ is a modular form over $\F_3$. Again by Corollary (5) in \textit{loc. cit.} and by Mazur \cite[Proposition 5.6 (b)]{M77}, we get $h=3^{-\alpha(3)}\times a_0\times e$; in particular, 
$3^{-\alpha(3)} \times a_1 \equiv 0 \pmod 3$, i.e., $a_1 = 0 \in \zmod {3^{\alpha(3)+1}}$ as desired.
\end{itemize}
(Note that in the first case, we can allow the case where $M=N$ by taking $E:=\cE_{M/{\ell}, N/{\ell}} \pmod {\ell^{\alpha(\ell)}}$, which is used in the proof of Theorem \ref{thm:index1} above.) Therefore we have $I=(\ell^{\alpha(\ell)},~I_{M, N/{\ell}})$ and
$$
\T(N)/\cI \simeq \T(N)^{\ell\hyp\old}/\cI \simeq R/{I} = \T(N/{\ell})/(\ell^{\alpha(\ell)},~I_{M, N/{\ell}}).
$$
Accordingly, it suffices to prove that $\alpha(\ell)\leq \beta(\ell)$ for primes $\ell$ not dividing $N/M$ because $\ell \nmid \ell^2-1$. 

Next, we assume that $\ell$ does not divide $N/M$.
Let $F:=\cE_{M, N} \pmod {24\ell^{\alpha(\ell)}}$ and $\delta$ be a cusp form as above.
Since $F$ and $-24\delta$ have the same $q$-expansions (at $P_N$), they coincide on the irreducible component $D$ of $X_0(N)_{\F_{\ell}}$, which contains $P_N$. Note that the cusp $P_{M}$ belongs to $D$ because $\ell \nmid N/M$. Since $-24\delta$ is a cusp form over the ring $\zmod {24\ell^{\alpha(\ell)}}$, the residue of $F$ at $P_M$ must be zero. By Proposition \ref{prop:residue}, $\varphi(N)\psi(N/M)(M/N) \equiv 0 \pmod {24\ell^{\alpha(\ell)}}$.
Therefore we get $\alpha(\ell)\leq \beta(\ell)$. (Note that if $\ell=2$, then $h=1$ with the assumption
that $M\neq N$ and $\ell \nmid N/M$.)
\end{proof}

If $\ell$ is odd and $\ell \nmid \varphi(N)$, we prove the following.

\begin{prop}\label{prop:nonmaximal}
Let $\ell$ be an odd prime and $\m:=(\ell,~I_{1, N})$. Hence, we assume that $\ell \nmid \varphi(N)$ from the definition (cf. \textsection \ref{sec:Eisenstein}).
Then, $\m$ cannot be maximal.
\end{prop}

\begin{proof}
Assume that $\m$ is maximal. If $\ell \mid N$, then $\m$ cannot be $\ell$-new because $U_{\ell} \equiv \ell \equiv 0 \pmod {\m}$. Therefore there is a maximal ideal $\fn:=(\ell,~I_{1, N/{\ell}})$ in the Hecke ring $\T(N/{\ell})$ of level $N/{\ell}$. Thus, we may assume that $\ell \nmid N$. 
Then as above, $\delta$ is a mod $\ell$ cusp form of weight $2$ and level $N$. Let $g=\cE_{N, N} \pmod {24\ell}+24\delta$ be a modular form over $\zmod{24\ell}$.

First, consider the case where $n=\omega(N)=1$. 
\begin{itemize}
\item
If $\ell\geq 5$, then $g$ is a mod $\ell$ modular form of weight $2$ and level $N$ as above. Since the $q$-expansion of $g$ is 
$$
(1-N) + 24(1-N) \sum\limits_{i=1}^\infty \sigma(d) \times q^{dN},
$$
we get $\frac{g}{1-N}=0$ by Mazur \cite[chap. II, Corollary 5.11]{M77}, which is a contradiction. Therefore $\m$ is not maximal. 

\item
If $\ell=3$, then $g$ is a modular form of weight $2$ and level $N$ over $\zmod 9$ as above. Then, by Mazur \cite[chap. II, Lemma 5.9]{M77}, there is a modular form $G$ of level $1$ over $\zmod 9$ such that $G(Nz)=\frac{g(z)}{1-N}$. However this contradicts Proposition 5.6(c) in \cite[chap. II]{M77}. Therefore $\m$ is not maximal.
\end{itemize}

Next, consider the case where $n\geq 2$. Let $F_N(q): = (-1/24) \times \cE_{1, N} \in \Z[[q]]$ be a formal $q$-expansion. Since $\m$ is maximal, $\delta \equiv F_N(q) \modl$ is a mod $\ell$ modular form of weight $2$ and level $N$.
Then, by the following lemma, we can lower the level of $\delta$ because $\varphi(N)\not\equiv 0 \modl$. Therefore the result follows from the case where $n=1$.

\end{proof}

\begin{lem}
Let $N=pD$ be a square-free integer with $D>1$ and $p$ a prime.
Assume that $p \not\equiv 1 \modl$ and $\ell \nmid N$.  
Let $F_N(q) : = (-1/24) \times \cE_{1, N} \in \Z[[q]]$ be a formal $q$-expansion. If $F_{N}(q) \modl$ is the $q$-expansion of a mod $\ell$ modular form of weight $2$ and level $N$, then 
$F_{D}(q) \modl$ is also the $q$-expansion of a mod $\ell$ modular form of weight $2$ and level $D$.
\end{lem}
\begin{proof}
Let $G(q):=(-1/24) \times \cE_{p, N}$. Then, as formal $q$-expansions we get
$$
F_N(q)-G(q) = (p-1)F_D(q^p).
$$
Therefore if $F_N(q) \modl$ is the $q$-expansion of a mod $\ell$ modular form of level $N$, then there is a mod $\ell$ modular form of level $D$ whose $q$-expansion is $(p-1)F_D(q) \modl$ by Ohta \cite[Lemma (2.1.1)]{Oh14}. Therefore the result follows because $p\not\equiv 1 \modl$.
\end{proof}

\section{Proof of the main theorem}\label{sec:proof}
In this section, we prove our main theorem.
\begin{thm}\label{thm:maintheorem}
Let $\m:=(\ell,~I_{M, N})$ be a maximal ideal of $\T(N)$. Then $\cC_N[\m] \neq 0$.
\end{thm}

\begin{proof}
If $\ell$ is odd, then the result follows from Theorems \ref{thm:index1} and  \ref{thm:index2}. Therefore we assume that $\ell=2$. By the definition of the notation, $M$ is either $N$ or $N/2$.

\begin{itemize}
\item 
If $N$ is a prime and $N=M$, then $M\equiv 1 \pmod 8$ by Mazur \cite{M77}. Thus, we have $\cC_N[\m] \neq 0$.

\item 
If $N$ is not a prime and $N=M$, then we set $N=pD$ with $D$ odd and $\omega(D)\geq 1$.
(In other words, if $N$ is even then we set $p=2$.)
Since $(2,\, I_{N, N})=(2, \, I_{p, N})$ is maximal,  
the index of $I_{p, N}$, which is equal to the order of $C_{p, N}$, is divisible by $2$ 
and hence $\br {C_{p, N}}[\m] \neq 0$, which implies that $\cC_N[\m] \neq 0$.

\item
If $N=2M$ with $\omega(M)=1$, then $\m$ is not $2$-new and hence there is the corresponding Eisenstein maximal ideal of $\T(M)$. Therefore $M\equiv 1 \pmod 8$ by Mazur. This implies that the order of $C_{M, N}$ is $\frac{M-1}{4}$ by Theorem \ref{thm:ordercuspidal}. Thus, we get $\cC_N[\m]\neq 0$.

\item
If $N=2M$ with $\omega(M)\geq 2$, then the order of $C_{p, N}$ is divisible by $2$, where $p$ is any prime divisor of $M$. Therefore we get $\cC_N[\m] \neq 0$.
\end{itemize}
\end{proof}

\bibliographystyle{annotation}

\end{document}